\documentclass[aap,noinfoline,preprint]{imsart}

\usepackage{a4wide}
\usepackage{graphics}
\usepackage{enumitem}
\usepackage{booktabs}
\usepackage{diagbox}
\usepackage{calc}
\usepackage{microtype}
\usepackage{indentfirst}
\usepackage{url}

\usepackage{mathtools}
\usepackage{amsthm}
\usepackage{nicematrix}
\usepackage{arydshln}
\usepackage{amssymb}
\usepackage{accents}
\usepackage{bm}

\allowdisplaybreaks

\newsavebox\mybox
\newcommand{\fphantom}[2]{
	\sbox{\mybox}{$\mathsurround=0pt #2$}
	\hspace*{#1\wd\mybox}
}

\newcommand{\alignas}[3]{
	\mathmakebox[\widthof{$\mathsurround=0pt #1$}][#2]{#3}
}

\newcommand{\intercenter}[2]{%
	\multispan{#1}%
	\hfill$\mathsurround=0pt \displaystyle{\mathclap{#2}}$\hfill
	\ignorespaces
}

\DeclarePairedDelimiter{\inparen}{\lparen}{\rparen}
\DeclarePairedDelimiter{\inbrace}{\lbrace}{\rbrace}
\DeclarePairedDelimiter{\abs}{\lvert}{\rvert}
\DeclarePairedDelimiter{\floor}{\lfloor}{\rfloor}

\newcommand{\set}[2][]{
	\inbrace*{ \mathop{}
		\ifx #1\relax \relax
		#2
		\else	
		#1 \mathrel{} \middle\vert \mathrel{} #2
		\fi
	\mathop{} }
}

\def\>={\geqslant}
\def\<={\leqslant}
\newcommand{\defas}{\coloneqq}
\newcommand{\Z}{\mathbb{Z}}
\newcommand{\R}{\mathbb{R}}
\newcommand{\C}{\mathbb{C}}
\renewcommand{\d}{{\operatorname{d}}}
\renewcommand{\i}{\mathrm{i}}
\newcommand{\e}{\mathrm{e}}

\newsavebox{\foobox}
\newcommand{\slantbox}[2][0]{\mbox{
		\sbox{\foobox}{#2}
		\hskip\wd\foobox
		\pdfsave
		\pdfsetmatrix{1 0 #1 1}
		\llap{\usebox{\foobox}}
		\pdfrestore
}}
\newcommand\unslant[2][-.25]{\slantbox[#1]{$#2$}}
\let\slpi\pi
\renewcommand{\pi}{
	\alignas{\slpi}{l}{\fphantom{.125}{\slpi}\unslant\slpi}
}

\newcommand{\cin}{\alignas{{}\cos{}}{c}{{}\sin{}}}
\renewcommand{\vec}[1]{\bm{#1}}
\newcommand{\mat}[1]{\mathbf{#1}}
\newcommand{\sqrdim}[1]{
	_{\mathrlap{\!\!\!#1 \times #1}}
}
\def\T{{\mkern-1.5mu\mathsf{T}}}

\renewcommand{\:}{\,:\,}

\def\n4+{
	n_{\alignas{_\textnormal{T}}{c}{
		4\hss\scalebox{0.69444}{$\mathsurround=0pt ^+$}
	}}
}
\def\hatn4+{
	\alignas{\n4+}{l}{
		\widehat{\alignas{n_4}{l}{\n4+}}
	}
}
\newcommand{\nT}{n_\textnormal{T}}
\def\H2+{H_{2^+}\!}

\DeclareMathAccent{\wtilde}{\mathord}{largesymbols}{"65}
\let\underscore\_
\def\B_~{
	\underaccent{\wtilde}{\mat{B}}
}

\startlocaldefs
\numberwithin{equation}{section} \theoremstyle{plain}
\newtheorem{thm}{Theorem}[section]
\newtheorem{lem}{Lemma}[section]

\endlocaldefs

\setattribute{journal}{name}{}

\begin{document}

\begin{frontmatter}

  \title{On Laplacian spectrum of dendrite trees}
    
  \runtitle{On Laplacian spectrum of dendrite trees}

  \begin{aug}
    \author{\fnms{~ Yuyang} \snm{Xu}\ead[label=e1]{xuyy@connect.hku.hk}}
    \and
    \author{\fnms{~ Jianfeng} \snm{Yao~}\ead[label=e2]{jeffyao@hku.hk}}
    
    \affiliation{The University of Hong Kong}
    \runauthor{Y. Xu \&    J. Yao}

    \address{ 
      Department of Statistics and Actuarial Science\\
      The University of Hong Kong\\
      \printead{e1,e2}
    }
  \end{aug}

  \begin{abstract}
    For dendrite graphs from biological experiments on mouse's
    retinal ganglion cells, a paper by Nakatsukasa, Saito and Woei
    reveals a mysterious phase transition phenomenon in the spectra
    of the corresponding graph Laplacian matrices. While the bulk of
    the spectrum can be well understood by 
    structures resembling 
    starlike trees,  mysteries about the spikes, that is, isolated
    eigenvalues outside the bulk spectrum, remain unexplained.   In this
    paper, we bring new insights on these mysteries by considering a
    class of uniform trees.  Exact relationships between the number of
    such spikes and the number of T-junctions are analyzed in function
    of the number of vertices separating the T-junctions. Using these
    theoretical results, predictions are proposed for the number of spikes
    observed in real-life 
    dendrite graphs. Interestingly enough, these predictions match well the observed
    numbers of spikes, thus confirm the practical meaningness of our
    theoretical results. 
  \end{abstract}
  
  \begin{keyword}[class=AMS]
    \kwd[Primary ]{15A18;~}
    \kwd[Secondary ]{05C05}
  \end{keyword}

  \begin{keyword}
    \kwd{Dendrite graphs}
    \kwd{Eigenvalue distribution}
    \kwd{Graph Laplacian}
    \kwd{Retinal ganglion cells}
    \kwd{Number of spikes}
    \kwd{Spike eigenvalues}
  \end{keyword}

\end{frontmatter}

\section{Introduction}

Dendrite graphs are an important data analysis tool for understanding
biological functions of the retinal ganglion cells (RGCs) of  a mouse
\cite{Coombs06}. Characteristic features of these graphs such as dendritic field area,
dendrite lengths, or number of dendrites
can be used to classify various RGCs into different groups which do 
have different biological functions.
The details about data acquisition and their conversion to
dendrite graphs, which are almost trees, can be found
in~\cite{Saito_2009}. Figure~\ref{fig: RGC & eigs distributions}
illustrates two typical dendrite graphs of RGCs together with a plot
of eigenvalues of their graph Laplacian.
When Nakatsukasa et al.~\cite{Nakatsukasa_2013} analyzed the spectrum of the graph
Laplacian of these dendrite graphs, they revealed a ``mysterious''
phase  transition phenomenon in the spectrum.
The eigenvalue distribution consists of a
smooth bell-shaped curve ranging from~$0$ to~$4$, the {\em bulk\/},
followed by a sudden jump at~$4$ and a few {\em spikes\/}, i.e., isolated
eigenvalues greater than $4$. In order to understand this phenomenon,
they considered the family of starlike trees, i.e., trees with
exactly one vertex of degree greater than $2$, as a simplified model for
these complex dendrite graphs. Actually,  the bulk of the dendrite
spectra is well approximated 
by such starlike trees. However, the mechanism that generates those
spike eigenvalues remains unknown.

\begin{figure}[tp]
	\centering
	\includegraphics[width=\textwidth]{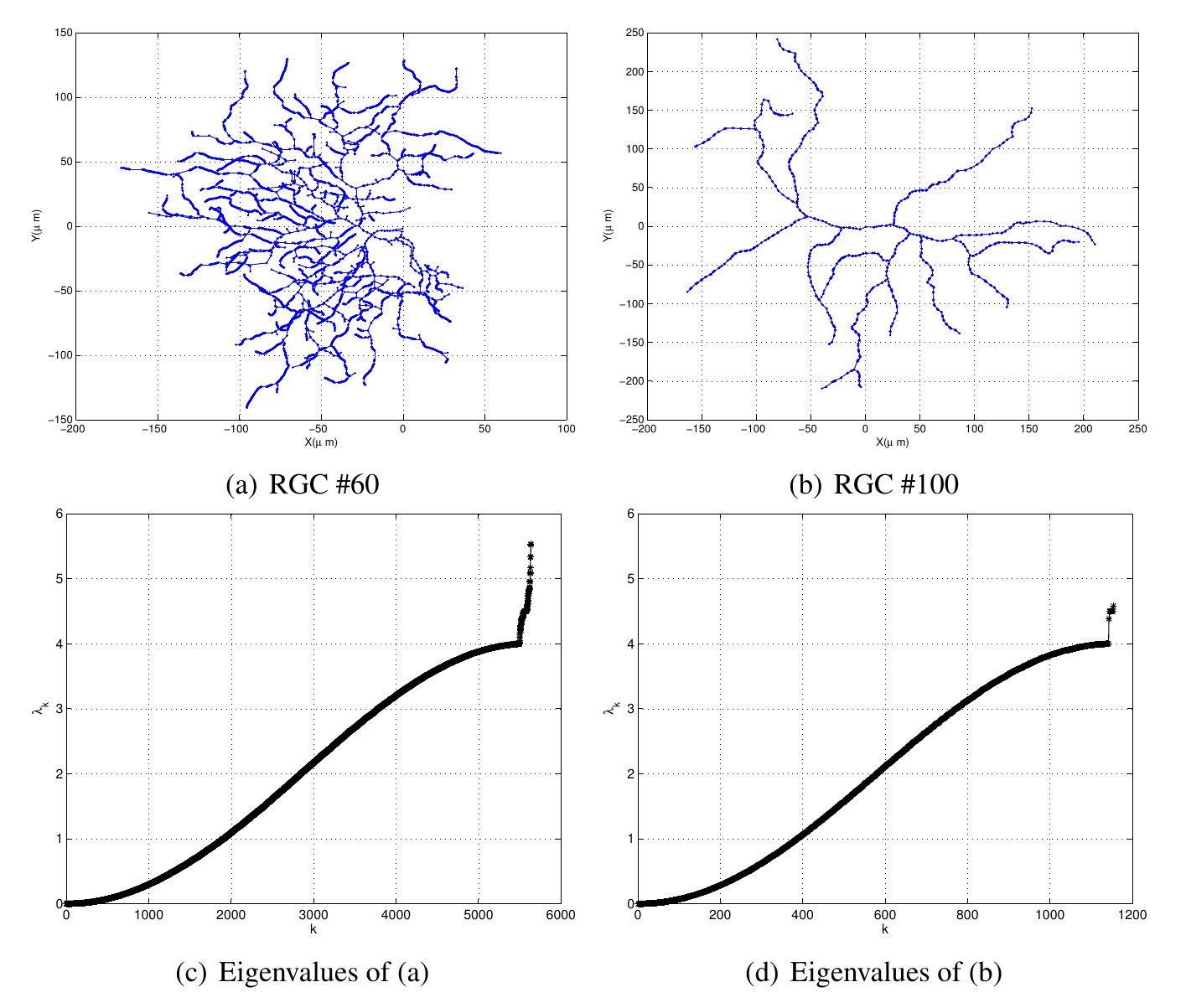}
	\caption{\textup{(a)} and \textup{(b):} typical 2D projection of dendrites of retinal ganglion
      cells \textup(RGCs\textup)  of a mouse\textup; \textup{(c)} and \textup{(d):} eigenvalue distribution
      of corresponding graph Laplacians. Both graphs show a phase transitions occur at
      the value\/~$4$. Data and code used for producing the figure by courtesy of Prof.~Naoki Saito
      \textup(University of California, Davis\textup).
      \label{fig: RGC & eigs distributions}}
\end{figure}

This paper shed new lights on the ``mysteries around the graph
Laplacian eigenvalue $4$'' found in \cite{Nakatsukasa_2013}.  We observe that dendrite trees consist of a
number of T-junctions, whose degree are exactly $3$, a lot of path
vertices, whose degree are exactly $2$, and pendant vertices, whose
degree are~$1$. As a more involved model than starlike trees, we
consider a class of uniform trees. Precisely, 
let $m \>= 0$ and $k \>= 1$ be nonnegative integers.
A tree with maximum degree of $3$ is a {\em uniform tree of type $(m,k)$\/}, denoted as $H_{m,k}$, 
if
\begin{enumerate}
\item the tree has a path lining a number of T-junctions (vertices
  of degree $3$) where any two adjacent T-junctions are separated by
  a segment made with a given number of $m$ vertices, called
  {\em trunk\/};
\item
  from each T-junction departs a {\em branch\/} which is a simple
  path with $k$ vertices terminating at
  a pendant vertex with degree $1$.
\end{enumerate}
Some examples of uniform trees are given in Figure~\ref{fig: uniform tree}.
\begin{figure}
  \centering
  \includegraphics[width = \textwidth]{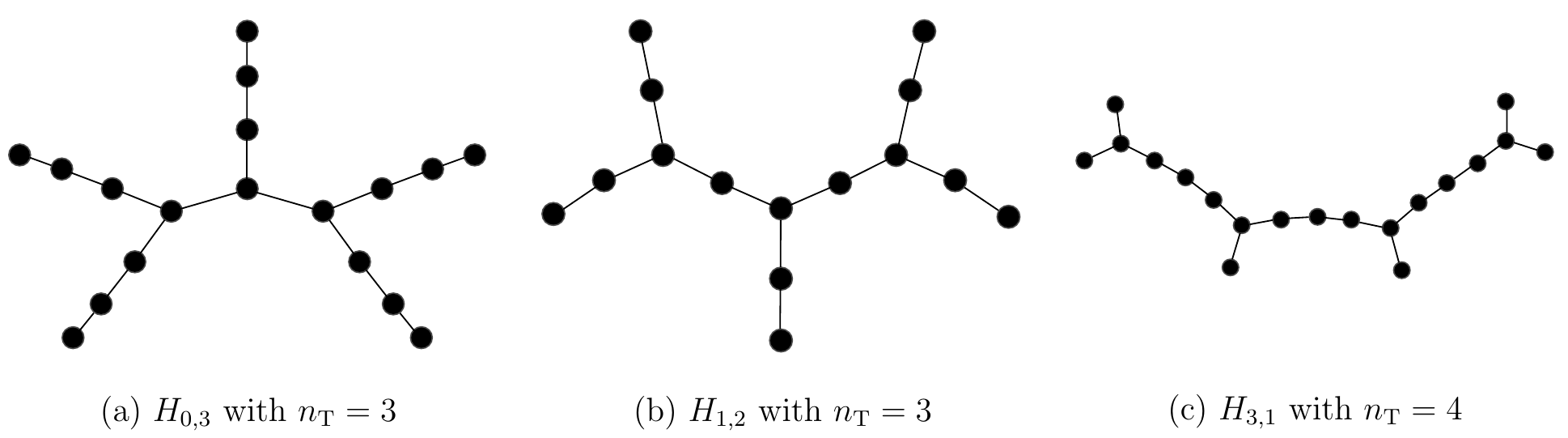}
  \caption{Examples of uniform trees\/ $H_{m,k}$ with trunk
    length\/ $m$ and branch length\/ $k$.
  \label{fig: uniform tree}}
\end{figure}

Our research starts from some numerical experiments on such uniform
trees. To explain empirical findings  from these numerical experiments,
let us first introduce some notations.
Consider  an arbitrary  graph  $G = (V,E)$  with $n$
vertices in $V$  and edge set $E$.
Let  $\mat{L}(G)$ be  its  Laplacian matrix
with eigenvalues  $\lambda_1(G) \>=
\dotsb \>= \lambda_j(G) \>= \dotsb \>= \lambda_n(G) = 0$.
Define 
\begin{description}[labelsep = \parindent, leftmargin = \parindent + \widthof{$m\>=0$}, parsep=0pt]
\item $\n4+(G) \defas \#\set[\lambda_j(G)]{\lambda_j(G) \>= 4,\ 1\<=j\<=n}$, i.e., the number
  of eigenvalues of $\mat{L}(G)$ at least equal to 4;
\item $\nT(G) \defas \#\set[v_j \in V]{\deg(v_j) = 3,\ 1\<=j\<=n} \>= 2$, the number of T-junctions in~$G$.
\end{description}

\smallskip
Our empirical findings of uniform trees $H_{m,k}$ are as follows:
\begin{enumerate}
\item[(a)] For any $H_{m,k}$ with trunk length $m \>= 2$,  $\n4+(H_{m,k}) = \nT(H_{m,k})$ for all $k \>= 1$;
\item[(b)] For $H_{m,k}$ with trunk length $m=1$ and any fixed
T-junction number ${ \nT(\cdot) \equiv \nT(H_{m,k})}$, $\n4+(H_{m,k}) \nearrow \nT$ as $k \to \infty$;  \smallskip
\item[(c)] For $H_{m,k}$ with trunk length $m=0$ and any fixed
T-junction number ${ \nT(\cdot) \equiv \nT(H_{m,k})}$, $\n4+(H_{m,k}) \nearrow \floor*{ \dfrac{\nT+1}{2} }$ as $k \to \infty$. \smallskip
\end{enumerate}

The convergences in (b) and (c) above can be verified with the
numerical results given in 
Tables~\ref{tab: m=1} and \ref{tab: m=0}, respectively. 

\begin{table}
	\centering
	\caption{Uniform trees with constant trunk length\/ $m=1$\textup: numerical values of\/ 
      $n_{\hspace{-1pt}4^{\hspace{-1pt}+\!}}(H_{1,k})$ for various
      branch lengths\/ $k$ and \textup{T}-junction numbers\/ $\nT$.}
	\begin{tabular}{c|cccccccccccccccc}
		\diagbox{$k$}{$\n4+(H_{1,k})$}{$\nT$} & 2 & 3 & 4 & 5 & 6 & 7 & 8 & 9 & 10 & 15 & 20 & 25 & 30 & 40 & 50 &  \\
	\hline
		            \phantom{00}1             & 1 & 2 & 2 & 3 & 4 & 4 & 5 & 5 & 6  & 9  & 12 & 15 & 18 & 24 & 30 &  \\
		            \phantom{00}2             & 2 & 2 & 3 & 4 & 5 & 5 & 6 & 7 & 7  & 11 & 14 & 18 & 21 & 28 & 36 &  \\
		            \phantom{00}3             &   & 3 & 3 & 4 & 5 & 6 & 6 & 7 & 8  & 12 & 15 & 19 & 23 & 31 & 38 &  \\
		            \phantom{00}4             &   &   & 4 & 4 & 5 & 6 & 7 & 7 & 8  & 12 & 16 & 20 & 24 & 32 & 40 &  \\
		            \phantom{00}5             &   &   &   & 4 & 5 & 6 & 7 & 8 & 9  & 13 & 17 & 21 & 25 & 33 & 41 &  \\
		            \phantom{00}6             &   &   &   & 5 & 5 & 6 & 7 & 8 & 9  & 13 & 17 & 21 & 25 & 33 & 42 &  \\
		            \phantom{00}7             &   &   &   &   & 5 & 6 & 7 & 8 & 9  & 13 & 17 & 21 & 25 & 34 & 42 &  \\
		            \phantom{00}8             &   &   &   &   & 6 & 6 & 7 & 8 & 9  & 13 & 17 & 22 & 26 & 34 & 43 &  \\
		            \phantom{00}9             &   &   &   &   &   & 6 & 7 & 8 & 9  & 13 & 18 & 22 & 26 & 35 & 43 &  \\
		            \phantom{0}10             &   &   &   &   &   & 7 & 7 & 8 & 9  & 13 & 18 & 22 & 26 & 35 & 43 &  \\
		            \phantom{0}15             &   &   &   &   &   &   & 8 & 8 & 9  & 14 & 18 & 23 & 27 & 36 & 45 &  \\
		            \phantom{0}20             &   &   &   &   &   &   &   & 9 & 9  & 14 & 18 & 23 & 27 & 37 & 46 &  \\
		            \phantom{0}30             &   &   &   &   &   &   &   &   & 10 & 14 & 19 & 23 & 28 & 37 & 46 &  \\
		            \phantom{0}50             &   &   &   &   &   &   &   &   &    & 15 & 19 & 24 & 29 & 38 & 47 &  \\
		                 100                  &   &   &   &   &   &   &   &   &    &    & 20 & 24 & 29 & 39 & 48 &  \\
		                 150                  &   &   &   &   &   &   &   &   &    &    &    & 25 & 29 & 39 & 49 &  \\
		                 200                  &   &   &   &   &   &   &   &   &    &    &    &    & 30 & 39 & 49 &  \\
	\hline
	\end{tabular}
	\label{tab: m=1}
\end{table}

\begin{table}
	\centering
	\caption{Uniform trees with constant trunk length\/ $m=0$\textup: numerical values of\/
      $n_{\hspace{-1pt}4^{\hspace{-1pt}+\!}}(H_{0,k})$ for various
      branch lengths\/ $k$ and \textup{T}-junction numbers\/ $\nT$.}
	\begin{tabular}{c|cccccccccccc}
		\diagbox{$k$}{$\n4+(H_{0,k})$}{$\nT$} & 2 & 3 & 4 & 5 & 6 & 7 & 8 & 9 & 10 & 15 & 30 & 50 \\
	\hline
		            \phantom{00}1             & 1 & 1 & 2 & 2 & 2 & 3 & 3 & 4 & 4  & 6  & 12 & 20 \\
		            \phantom{00}2             & \smash{\vdots} & 2 & 2 & 2 & 3 & 3 & 4 & 4 & 5  & 7  & 13 & 22 \\
		            \phantom{00}3             & \smash{\vdots} & \smash{\vdots} & \smash{\vdots} & 3 & 3 & 4 & 4 & 4 & 5  & 7  & 14 & 23 \\
		            \phantom{00}4             & \smash{\vdots} & \smash{\vdots} & \smash{\vdots} & \smash{\vdots} & \smash{\vdots} & \smash{\vdots} & \smash{\vdots} & 5 & 5  & 7  & 14 & 24 \\
		            \phantom{00}5             & \smash{\vdots} & \smash{\vdots} & \smash{\vdots} & \smash{\vdots} & \smash{\vdots} & \smash{\vdots} & \smash{\vdots} & \smash{\vdots} & \smash{\vdots}  & 7  & 15 & 24 \\
		            \phantom{00}6             & \smash{\vdots} & \smash{\vdots} & \smash{\vdots} & \smash{\vdots} & \smash{\vdots} & \smash{\vdots} & \smash{\vdots} & \smash{\vdots} & \smash{\vdots}  & 8  & 15 & 24 \\
		            \phantom{00}7             & \smash{\vdots} & \smash{\vdots} & \smash{\vdots} & \smash{\vdots} & \smash{\vdots} & \smash{\vdots} & \smash{\vdots} & \smash{\vdots} & \smash{\vdots}  & \smash{\vdots}  & \smash{\vdots} & 24 \\
		            \phantom{00}8             & \smash{\vdots} & \smash{\vdots} & \smash{\vdots} & \smash{\vdots} & \smash{\vdots} & \smash{\vdots} & \smash{\vdots} & \smash{\vdots} & \smash{\vdots}  & \smash{\vdots}  & \smash{\vdots} & 25 \\
		            \phantom{0}20             & \smash{\vdots} & \smash{\vdots} & \smash{\vdots} & \smash{\vdots} & \smash{\vdots} & \smash{\vdots} & \smash{\vdots} & \smash{\vdots} & \smash{\vdots}  & \smash{\vdots}  & \smash{\vdots} & \smash{\vdots} \\
		                 100                  & 1 & 2 & 2 & 3 & 3 & 4 & 4 & 5 & 5  & 8  & 15 & 25 \\
	\hline
	\end{tabular}
	\label{tab: m=0}
\end{table}

The main contribution of this paper is a rigorous proof for the three
empirical findings (a)-(b)-(c) above.

These theoretical results on uniform trees provide a way to explain the
spectra of the graph Laplacians of real-life dendrite trees as those
shown in Figure~\ref{fig: RGC & eigs distributions}. Indeed, these
dendrite graphs can be viewed as a mixture of many uniform trees of
different types, that is, with different trunk and branch lengths.
Using the results (a)-(b)-(c),  we provide in Section~\ref{sec:RGC} an  estimation method  for the
number of spikes in term of different T-junctions and corresponding
trunk lengths observed in these dendrite graphs.  Estimates resulting
from this method are thus compared to
the empirical values found  in the dendrite
graphs. Interestingly enough, the estimates match the empirical values
very well. 

The remaining of the paper is organized as follows.
Section~\ref{sec:prelim} recalls some known and useful  results on Laplacian
eigenvalues of trees  and 
matrix perturbation.
The main results of the paper are introduced in Section~\ref{sec:main}.
As mentioned, we apply these theoretical results on {uniform trees} to
estimate the number of spikes in dendrite graphs in Section~\ref{sec:RGC}.

\section{Preliminary Results}\label{sec:prelim}

A few known results on Laplacian
eigenvalues of starlike trees and matrix perturbation will be useful.
First recall that 
a \emph{starlike tree\/} is defined as a tree with exactly one vertex
of degree at least equal to~$3$.
So the smallest starlike tree is $K_{1,3}$, which is also known as a claw.

\begin{lem}[Y.~Nakatsukasa et al. \cite{Nakatsukasa_2013}] \label{lem: starlike tree}
  A starlike tree has exactly one Laplacian eigenvalue greater than or equal to\/~$4$. The equality holds if and only if the starlike tree is\/~$K_{1,3}$.
\end{lem}

There is a well-known upper bound for the number of Laplacian spikes for trees where vertices have degree $3$ at most. 
\begin{lem}[Y.~Nakatsukasa et al. \cite{Nakatsukasa_2013}] \label{lem: n4+ upper bound}
  For any finite tree\/~$G$ with maximum degree of\/~$3$, we have 
  \[
	\n4+(G) \<= \nT(G).
  \]
\end{lem}

Next is a very useful interlacing lemma for perturbation on Laplacian eigenvalues.
\begin{lem}[R.~Grone et al. \cite{Grone_1990}] \label{lem: interlacing theorem}
  Let\/~$G$ be a tree on\/~$n$ vertices. Suppose\/~$\widetilde{G}$ is
  a subtree of\/~$G$ obtained by removing exactly one edge, then the\/
  $n{-}1$ largest eigenvalues of\/ $\mat{L}(\widetilde{G})$ interlaces the eigenvalues of\/ $\mat{L}(G)$. That is,
  \[
  	\lambda_1(G) \>= \lambda_1(\widetilde{G}) \>= \lambda_2(G) \>= \dotsb \>= \lambda_{n-1}(G) \>= \lambda_{n-1}(\widetilde{G}) = \lambda_n(G) = 0.
  \]
\end{lem} 

An important consequence of this interlacing lemma on trees is as
follows. Consider a tree $G$ and a pendant vertex $s$ of $G$, that
is, the degree of $s$ is 1. Let $\widetilde{G}$ be the sub-tree
obtained from $G$ by removing $s$ and the edge terminating at $s$.
Then  $\n4+(\widetilde{G}) \<= \n4+({G})$.
Consequently  $\n4+( H_{m,k})$ is increasing in the branch length
$k$.
Taking into account the upper bound $\nT(G)$ in Lemma~\ref{lem: n4+
  upper bound}, if we let both the number of T-junctions $\nT(\cdot)$
and the trunk length
$m$ be fixed and the branch length $k$ tending to infinity, it holds that
\begin{equation}
	{\lim_{k\to\infty}}  \n4+( H_{m,k}) = \ell = \ell\bigl( m, \nT(\cdot) \bigr) \<= \nT(\cdot),
\end{equation}
for a limit $\ell$ depending on $m$ and $\nT(\cdot)$. 

The following result on tridiagonal matrices will be also useful.
\begin{lem}[W.~C.~Yueh et al. \cite{YUEH_2008}] \label{lem: eigs of tridiagonal matrices}
  Consider a tridiagonal matrix
  \[
	\mat{A}_n = \left(
      \begin{NiceArray}[nullify-dots]{RCCCL}
        \gamma + b & c &        &   &            \\
                 a & b & \Ddots &   &            \\
                   &   & \Ddots &   &            \\
                   &   & \Ddots & b & c          \\
                   &   &        & a & b + \delta
      \end{NiceArray}\right)\sqrdim{n},
  \]	
where\/ $a,b,c,\gamma,\delta \in \C$, $ac \neq 0$ and\/ $n \>= 2$. The eigenvalues {\textup(unordered\textup)} of\/~$\mat{A}_n$ are given by
\begin{gather}
  \lambda_j = b + 2\sqrt{a}\sqrt{c}\cos\theta_{\!j}, \qquad j = 1,2,\dotsc,n, \label{eq: eigs}
  \shortintertext{where\/ $\theta_{\!j} \in \C$ satisfy}
  \begin{dcases}
		ac\sin{(n{+}1)\theta_{\!j}} - \sqrt{a}\sqrt{c}(\gamma+\delta)\sin{n\theta_{\!j}} + \gamma\delta\sin{(n{-}1)\theta_{\!j}} = 0;
        & \textnormal{if } \sin\theta_{\!j} \neq 0, \\
		ac\cdot(n{+}1) - \cos\theta_{\!j}\sqrt{a}\sqrt{c}(\gamma+\delta)\cdot n + \gamma\delta\cdot(n{-}1) = 0;
			& \textnormal{if } \sin\theta_{\!j}   =  0.
	\end{dcases} \notag
\end{gather}
\end{lem} 

{In particular, when $a,b,c,\gamma,\delta$ are all real and $a=c$, $\mat{A}_n$ is real symmetric thus with real eigenvalues. By Equation~\eqref{eq: eigs}, we have all $\theta_{\!j} \in \R$.}

Also, note that for a path graph~$P_n$ with~$n$ vertices, i.e., a tree with maximum degree of~$2$, the Laplacian matrix is 
\[
	\mat{L}(P_n) = 
		\begin{pNiceMatrix}[nullify-dots]
			 1 & -1 &        &    &    \\
			-1 &  2 & \Ddots &    &    \\
			   &    & \Ddots &    &    \\
			   &    & \Ddots &  2 & -1 \\
			   &    &        & -1 &  1
		\end{pNiceMatrix}\sqrdim{n},
\]
which is a special case of~$\mat{A}_n$ in Lemma~\ref{lem: eigs of tridiagonal matrices}.
Applying the lemma we recover the well-known Laplacian eigenvalues of
$P_n$ given by 
\begin{equation} \label{eq: path graph eigs}
  \lambda_j(P_n) = 4 \sin^2\inparen*{\frac{n{-}j}{2n}\pi }, \qquad j = 1,2,\dotsc,n.
\end{equation}
They  are all less than~$4$.

Finally, we recall  the classical Schur complement formula.
\begin{lem}[Schur Complement Formula] \label{lem: recursion for alpha_k}
For an\/ $n \times n$ real matrix\/~$\mat{A}$, define\/ $\mat{A}_j,\
j=1,2,\dotsc,n$, to be the\/ $j$-th major submatrix of order\/
$n{-}1$, i.e., the submatrix resulting from\/~$\mat{A}$ by deleting its\/ $j$-th row and\/ $j$-th column.
Suppose both\/ $\mat{A}$ and\/~$\mat{A}_j$ are nonsingular, then
\[
	a^{jj} = \frac{1}{a_{jj} - \vec{\alpha}_j^\T \mat{A}_j^{-1} \vec{\beta}_j},
\]
where\/ $a^{jj}$ and\/~$a_{jj}$ are the\/ $j$-th diagonal entry of\/ $\mat{A}^{-1}$ and\/~$\mat{A}$, respectively\textup; $\vec{\alpha}_j^\T$ is the vector obtained from the\/ $j$-th row of\/~$\mat{A}$ by deleting its\/ $j$-th entry, and\/~$\vec{\beta}_j$ is the vector obtained from the\/ $j$-th column of\/~$\mat{A}$ by deleting its\/ $j$-th entry.
\end{lem} 

\section{Main Results}\label{sec:main}

Throughout this section, we consider a uniform tree with all its
T-junctions aligned on a single line.
Whenever no ambiguity is possible about  the branch length $k$, we simply use $H_m$ to denote $H_{m,k}$ for $m=0,1$ and $H_{2^+}$ for $m\>=2$.

\subsection{Uniform trees with trunk length\/ $m \>= 2$}
\begin{thm} \label{thm: m>=2} For uniform trees with trunk length $m \>= 2$, 
we have
\[
	\n4+(H_{m,k}) = \nT(H_{m,k}), \quad \forall k \>= 1.
\]
\end{thm}

\begin{proof}  The result is proved by induction on~$\nT$.
\begin{itemize}[label = $\lozenge$, labelsep = *, leftmargin = 0pt, itemindent = \parindent - (\widthof{\itshape{B\/}} - \widthof{\itshape{B}}), listparindent = \parindent]
\item \emph{Base Case}  \quad When $\nT=2$, let
  $\widetilde{H} = (V, E {\setminus} \inbrace{e})$, where $e$ is an
  edge on  the trunk  in~$\H2+$ that is not incident to any
  T-junction. The existence of such~$e$ is guaranteed by the
  assumption that $m \>= 2$, thus at least three edges lie between two
  consecutive T-junctions. Then $\widetilde{H}$ consists of two
  connected components, say $\widetilde{S}_1$ and~$\widetilde{S}_2$,
  which are both starlike trees. Figure~\ref{fig: m=2;nT=2} shows an
  example of $\H2+$ and~$\widetilde{H}$ with corresponding
  $\widetilde{S}_1$, $\widetilde{S}_2$ and an edge~$e$ removed  from $\H2+$.

\begin{figure}
	\centering
	\includegraphics[width = .95\textwidth]{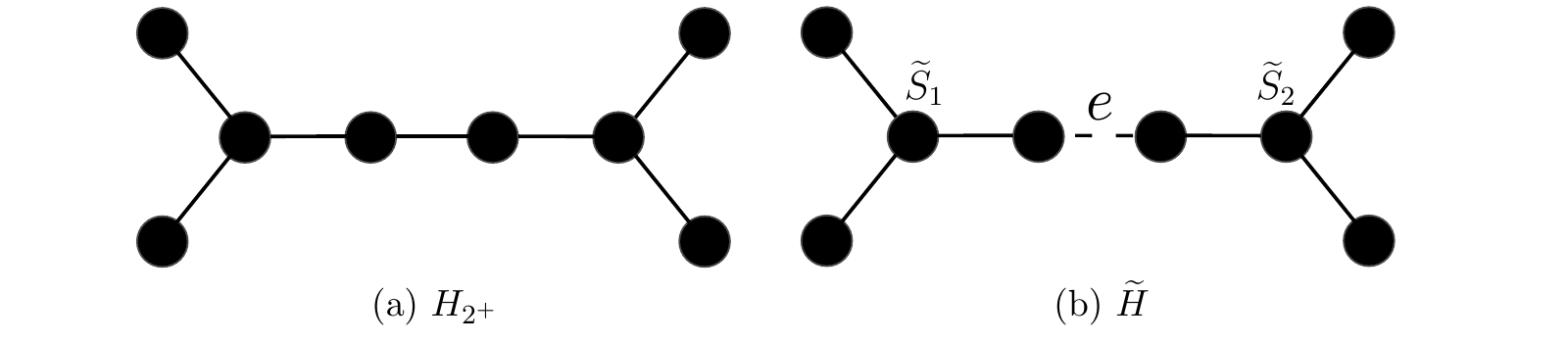}
	\caption{Illustration of base case in the proof of Theorem~\ref{thm: m>=2}. \textup{(a)}~$\H2+$ with\/ $\nT = 2$\textup; \textup{(b)}~$\widetilde{H}$ and its two connected components\/ $\widetilde{S}_1$, $\widetilde{S}_2$, obtained from\/~$\H2+$ by the removal of edge\/ $e$.}
	\label{fig: m=2;nT=2}
\end{figure} 
		
By Lemma~\ref{lem: starlike tree}, we have
\[
	\lambda_1(\widetilde{S}_1) \>= 4 > \lambda_2(\widetilde{S}_1), \quad \lambda_1(\widetilde{S}_2) \>= 4 > \lambda_2(\widetilde{S}_2).
\]
Without loss of generality, assume $\lambda_1(\widetilde{S}_1) \>= \lambda_1(\widetilde{S}_2)$. {Then we will have
\begin{equation}
	\label{eq: result}
	\lambda_1(\widetilde{H}) = \lambda_1(\widetilde{S}_1) \>= 4, \quad \lambda_2(\widetilde{H}) = \lambda_1(\widetilde{S}_2) \>= 4.
\end{equation}
Note that by Lemma~\ref{lem: interlacing theorem}, we have
\[
	\lambda_1(\H2+) \>= \lambda_1(\widetilde{H}) \>= \lambda_2(\H2+) \>= \lambda_2(\widetilde{H}).
\]
Using the result from~\eqref{eq: result}, we get}
\[
	\lambda_1(\H2+) \>= \lambda_1(\widetilde{S}_1) \>= \lambda_2(\H2+) \>= \lambda_1(\widetilde{S}_2) \>= 4.
\]
Hence, $\n4+ \>= 2$. However, by Lemma~\ref{lem: n4+ upper bound}, we have $\n4+ \<= \nT = 2$. Therefore, $\n4+ = \nT$ in this case.

\item \emph{Induction step}
\quad
Suppose $\n4+ = \nT$ holds for all $\nT \<= i$, for some $i \>= 2$.

When $\nT = i{+}1$, similarly, let $\widetilde{H} = (V, E {\setminus}
\{e\})$, where $e$ is an edge on a  trunk  in~$\H2+$ and is not incident to any T-junction. Let $\widetilde{H}_1$ and~$\widetilde{H}_2$ be the two connected components of~$\widetilde{H}$, with $t_1$ and~$t_2$ T-junctions, respectively. By inductive assumption we have
\begin{gather*}
	\lambda_1(\widetilde{H}_1) \>= \cdots \>= \lambda_{t_1}(\widetilde{H}_1) \>= 4 > \lambda_{t_1 + 1}(\widetilde{H}_1), \\
	\lambda_1(\widetilde{H}_2) \>= \cdots \>= \lambda_{t_2}(\widetilde{H}_2) \>= 4 > \lambda_{t_2 + 1}(\widetilde{H}_2).
\end{gather*}
Note that $t_1 + t_2 = \nT = i{+}1$. By Lemma~\ref{lem: interlacing theorem}, we have
\[
	\lambda_1(\H2+) \>= \lambda_1(\widetilde{H}) \>= \cdots \>= \lambda_{i+1}(\H2+) \>= \lambda_{i+1}(\widetilde{H}) \>= 4.
\]
Hence, $\n4+ \>= i{+}1$. However, by Lemma~\ref{lem: n4+ upper bound} we have $\n4+ \<= \nT = i{+}1$. Hence $\n4+ = \nT$. \qedhere
\end{itemize}
\end{proof}

\subsection{Uniform trees with trunk length\/ $m = 1$}
Unlike the previous situation with trunk length $m \>= 2$,  
numerical results  in Table~\ref{tab: m=1}  shows that $\n4+(H_{1,k}) \neq
\nT(H_{1,k})$ may happen when $m=1$. The same table also shows that 
$\n4+$ approaches $\nT$ as the branch length increases.

\begin{thm} \label{thm: m=1}
Consider uniform trees with trunk length\/ $m = 1$. When all the branches of\/~$H_1$ are sufficiently long, i.e., when\/ $k$ is large enough, we have
\[
	\n4+(H_{1,k}) = \nT(H_{1,k}).
\]
\end{thm}

\begin{proof}
For simplicity, we will use $t$ instead of $\nT$, and denote the Laplacian matrix of a uniform tree~$H_1$ by~$\mat{L}_1$. By considering first the main path connecting all the T-junctions and then the branch(es) on each T-junction, the Laplacian matrix $\mat{L}_1$ has the form (see Figure~\ref{fig: H1-L1})

\begin{figure}
	\centering
	\includegraphics[width = \textwidth]{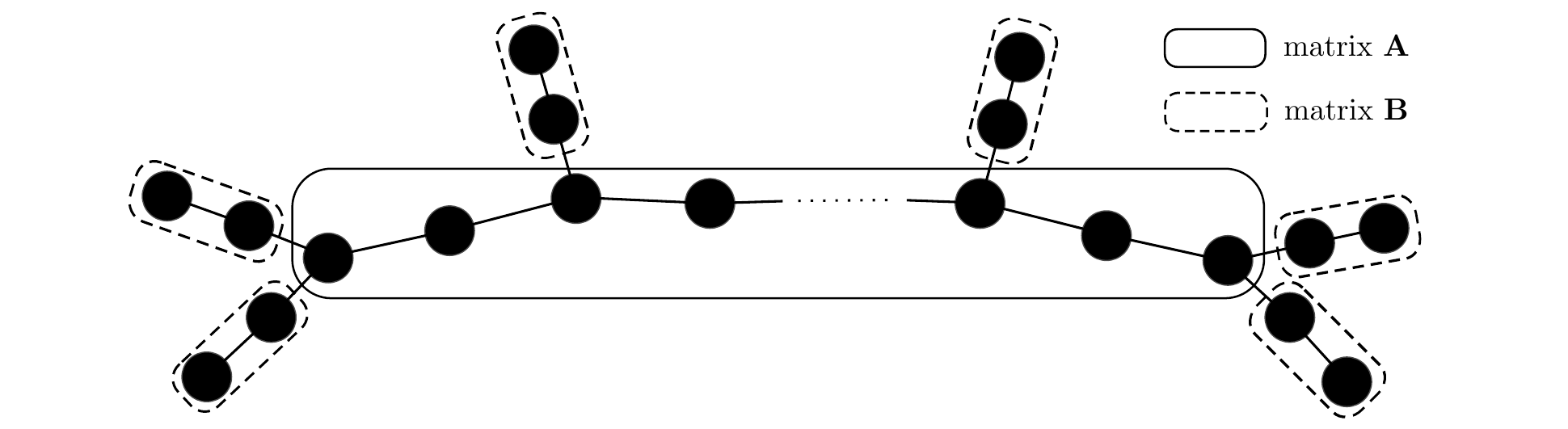}
	\caption{Illustration of uniform tree\/~$H_{1,k}$ and its counterpart in\/~$\mat{L}_1$, with\/ $\nT = 4$.}
	\label{fig: H1-L1}
\end{figure}

\begin{gather*} 
	\mat{L}_1 = \left(
		\begin{NiceArray}[nullify-dots]{C@{}C@{\ }CCL}
			 \mat{A}\phantom{_1} & \mat{E}_1^\T &          \mat{E}_2^\T              & \Cdots & \mat{E}_{t+2}^\T                  \\
			 \mat{E}         _1  &   \mat{B}    &                                    &        &                                   \\
			 \mat{E}         _2  &              & \alignas{\mat{E}_2^\T}{c}{\mat{B}} &        &                                   \\
			  \Vdots             &              &                                    & \Ddots &                                   \\
			\mat{E}_{t+2}        &              &                                    &        & \alignas{\mat{E}_2^\T}{c}{\mat{B}}
		\end{NiceArray}\right),
\shortintertext{where}
	\!\begin{aligned}
		\mat{A}\phantom{_1}
			&=	\begin{pNiceMatrix}[nullify-dots]
					 3 & -1 &    &        &    &    \\
					-1 &  2 & -1 &        &    &    \\
					   & -1 &  3 & \Ddots &    &    \\
					   &    &    & \Ddots &    &    \\
					   &    &    & \Ddots &  2 & -1 \\
					   &    &    &        & -1 &  3 
				\end{pNiceMatrix}\sqrdim{(2t-1)},{}
					& \mat{B}
						&= \left(\renewcommand{\arraystretch}{1.2}
							\begin{NiceArray}[nullify-dots]{
								 !{\fphantom{.5}{-}}
								C!{\phantom{-}}
								C!{\fphantom{.5}{-}}
								C!{\fphantom{.5}{-}}
								C!{\phantom{-}}
								C!{\fphantom{.5}{-}}
								}
								         2         & \alignas{2}{c}{-1} &        &                    &                    \\
								\alignas{2}{c}{-1} &                    & \Ddots &                    &                    \\
								                   &                    & \Ddots &                    &                    \\
								                   &                    & \Ddots &          2         & \alignas{1}{c}{-1} \\
								                   &                    &        & \alignas{2}{c}{-1} &          1
							\end{NiceArray}\right)\sqrdim{k}, \\
		\mat{E}         _1
			&=	\begin{pNiceMatrix}[nullify-dots] 
					-1      & 0      & \Cdots & 0      \\
					 0      &        &        &        \\
					 \Vdots & \Vdots &        & \Vdots \\
					 0      & 0      & \Cdots & 0      
				\end{pNiceMatrix},{} 
					& \mat{E}_{t+2}
						&=	\begin{pNiceMatrix}[nullify-dots]
								0      & \Cdots & 0      & -1      \\
								       &        &        &  0      \\
								\Vdots &        & \Vdots &  \Vdots \\
								0      & \Cdots & 0      &  0      
							\end{pNiceMatrix}, \bigskip\\\\
		\mat{E}         _i
			&= {\begin{pNiceMatrix}[nullify-dots]
					0      & \Cdots & 0      & \mathclap{\smash{\stackrel{(2i-3)\text{-th column}}{\stackrel{\vee}{-1}}}} & 0      & \Cdots & 0      \\
					       &        &        & 0                                                                          &        &        &        \\
					\Vdots &        & \Vdots & \Vdots                                                                     & \Vdots &        & \Vdots \\
					0      & \Cdots & 0      & 0                                                                          & 0      & \Cdots & 0      
				\end{pNiceMatrix}}, \mathrlap{\quad i = 2, 3, \dotsc, t{+}1.}
		\end{aligned}
\end{gather*}

\noindent We can also write~$\mat{L}_1$ in a more compact form, namely
\begin{gather*} 
	\mat{L}_1 = \left(
		\begin{array}{@{}cl@{}}
			\mat{A} & \mat{E}^\T \\
			\mat{E} & \B_~     
		\end{array}\right),
\shortintertext{where}
	\B_~ = \mat{B} \otimes \mat{I}_{t+2} =
		\begin{pNiceMatrix}[nullify-dots]
			\mat{B} &         &         \\
			        &  \Ddots &         \\
			        &         & \mat{B} 
		\end{pNiceMatrix}, \quad \mat{E} =
			\begin{pNiceMatrix}[nullify-dots]
		  		\mat{E}_1   \\
		 		   \Vdots    \\
				\mat{E}_{t+2}
			\end{pNiceMatrix}.
\end{gather*}

\begin{enumerate}[label = \textbullet\ \emph{Step\/} \arabic*:, labelsep = *, leftmargin = 0pt, itemindent = *, listparindent = \parindent, itemsep = \bigskipamount]
\item block decomposition of the determinant.
Using Schur's determinant identity, any eigenvalue $\lambda$  of the Laplacian
$\mat{L}_1$ $\bigl($this is also $\lambda(H_1) \bigr)$ solves the equation
\begin{align*}
	0=\abs*{\lambda\mat{I} - \mat{L}_1}
		&=	\begin{vmatrix}
				\lambda\mat{I} - \mat{A} &     -\mat{E}^\T       \\
				        -\mat{E}         & \lambda\mat{I} - \B_~ 
			\end{vmatrix} \\
		&  =    \abs*{\lambda\mat{I} - \B_~} \cdot \abs*{\lambda\mat{I} - \mat{A} - \mat{E}^\T(\lambda\mat{I}{-}\B_~)^{-1}\mat{E}} \\
		&  =    \abs*{\lambda\mat{I} - \mat{B}}^{t+1} \cdot \Bigl| \lambda\mat{I} - \mat{A} -
			\sum_{i=1}^{t+2} \mat{E}_i^\T (\lambda\mat{I}{-}\mat{B})^{-1} \mat{E}_i \Bigr| \\
		&\equiv \abs*{\lambda\mat{I} - \mat{B}}^{t+1} \cdot \abs*{\lambda\mat{I} - \mat{A} - \mat{C}} .
\end{align*}
Here the factorization is valid for any $\lambda$ such that $\lambda\mat{I} - \B_~ \neq 0$. Because by Lemma~\ref{lem: eigs of tridiagonal matrices}, we can obtain the eigenvalues of $\mat{B}$, i.e.,
\[
	\lambda_j = 2 + 2\cos\frac{2j\pi}{2k+1}, \qquad j=1,2,\dotsc,k,
\]
which are all less than $4$. Therefore, to find those spike
eigenvalues $\lambda(H_1) \>= 4$, it suffices to consider roots of equation $\abs*{\lambda\mat{I} - \mat{A} - \mat{C}} = 0$. 
	
To further expand $\mat{C}$, let $\vec{\varepsilon}_l$ be the $l$-th canonical vector of $\R^{2t-1}$, then we have the decomposition 
\begin{align*}
	\intercenter{2}{\mat{E}_i = -\vec{\varepsilon}_{\sigma_{i}}\vec{\mu}_1^\T,}
\shortintertext{where}
	\vec{\mu}_1
		&= (1,0,\dotsc,0)^\T \in \R^k, \\
\shortintertext{and}
	\alignas{\vec{\mu}_1}{c}{\sigma_i}
		&=	\begin{dcases}
				1,{}                      & i  =  1           \\
				2         i       {-}3,{} & 2 \<= i \<= t{+}1 \\
				2\alignas{i}{c}{t}{-}1,{} & i  =  t{+}2
			\end{dcases}
\end{align*}
is the column position of~$-1$ in~$\mat{E}_i$. Therefore,
\begin{align*}
	\mat{C}
		&= \sum_{i=1}^{t+2} \vec{\varepsilon}_{\sigma_{i}} \inparen*{\vec{\mu}_1^\T
			(\lambda\mat{I}{-}\mat{B})^{-1}\vec{\mu}_1} \vec{\varepsilon}_{\sigma_{i}}^\T \\
		&= \inparen*{\vec{\mu}_1^\T (\lambda\mat{I}{-}\mat{B})^{-1} \vec{\mu}_1}
			\cdot \sum_{i=1}^{t+2} \vec{\varepsilon}_{\sigma_{i}} \alignas{\vec{e}}{l}{\vec{\varepsilon}_{\sigma_{i}}}^\T \\
		&= \alpha_k(\lambda) \cdot
			\begin{pNiceMatrix}[nullify-dots]
				2 &   &   &        &   &   \\
				  & 0 &   &        &   &   \\
				  &   & 1 &        &   &   \\
				  &   &   & \Ddots &   &   \\
				  &   &   &        & 0 &   \\
				  &   &   &        &   & 2 \\
			\end{pNiceMatrix}\sqrdim{(2t-1)},
\end{align*}
where $\alpha_k(\lambda) = \vec{\mu}_1^\T (\lambda\mat{I}{-}\mat{B})^{-1} \vec{\mu}_1$ is the top-left element of $(\lambda\mat{I} {-} \mat{B})^{-1}$. 

To see what $\alpha_k(\lambda)$ really is, let us consider the sequence $(\alpha_k)_{k\>=1} \equiv \bigr( \alpha_k(\lambda) \bigl)_{k\>=1}$. By Lemma~\ref{lem: recursion for alpha_k} we have the following recursion:
\[
	\begin{dcases}
		         \alpha_k               = \dfrac{1}{\lambda - 2 - \alpha_{k-1}}, & k \>= 1 \\
		\alignas{\alpha_k}{l}{\alpha_0} = -1
	\end{dcases}.
\]
Let
\[
	x = \frac{1}{2}\inparen*{\lambda - 2 - \sqrt{\lambda(\lambda {-} 4)}}, \qquad y = \frac{1}{2}\inparen*{\lambda - 2 + \sqrt{\lambda(\lambda {-} 4)}}.
\]
Then the solution to the recursion can be written in the following form,
\[
	\alpha_k =
		\begin{dcases}
			\frac{x^{k}y(x{+}1) - xy^k(y{+}1)}{x^k(x{+}1) - y^k(y{+}1)} \in \inparen*{0,\frac{2k-1}{2k+1}},{}
				& \lambda > 4 \\
			\frac{2k-1}{2k+1} < 1,{}
				& \lambda = 4
		\end{dcases}.
\]

\item let $k \rightarrow \infty$ to find a full  factorization of the
  limiting determinant.

With the preparation above, we can now let $k \rightarrow \infty$,
i.e., we investigate the uniform tree~$H_{1,k}$ when branch length
tends to infinity. This limiting tree is denoted by $H_{1,\infty}$. Our remaining task is to show
\begin{equation}
  \label{eq: claim}
  \#\set[\lambda_j]{\lambda_j(H_{1,\infty}) > 4,\ 1\<=j\<=n} = t.
\end{equation}
Suppose for a moment that this claim is true. By continuity of $H_{1,k} \mapsto \lambda_j(H_{1,k})$, the above equation implies existence of a $k_0$ such that for all $k' \>= k_0$, we have
\begin{align*}
	t	&= \#\set[\lambda_j]{\lambda_j(H_{1,\infty}) > 4 ,\ 1\<=j\<=n} \\
		&= \#\set[\lambda_j]{{\lim_{k\rightarrow\infty}}\lambda_j(H_{1,k}) > 4 ,\ 1\<=j\<=n} \\
		&= \#\set[\lambda_j]{\lambda_j(H_{1,k'}) > 4,\ 1\<=j\<=n} \\
		&\<= {\n4+(H_{1,k'})}.
\end{align*}
By Lemma~\ref{lem: n4+ upper bound}, ${\n4+(H_{1,k'})} \<= t$. Hence $\n4+(H_{1,k'}) = t$ and our theorem is proved.

It remains to establish the claim~\eqref{eq: claim}. Note that 
when $k \rightarrow \infty$,
\begin{align}
	\intercenter{2}{
		\alpha_k \rightarrow \alpha \defas \frac{1}{2}\inparen*{\lambda - 2 - \sqrt{\lambda(\lambda {-} 4)}} \in (0,1], \qquad  \lambda \>= 4.
	} \notag
\shortintertext{So we have}
	{\lim_{k\rightarrow\infty}} \det\inparen*{\lambda\mat{I} - \mat{A} - \mat{C}}
		&= \det\left(\renewcommand{\arraystretch}{1.2}
			\begin{NiceArray}[nullify-dots]{
				L@{\phantom{\lambda-}}
				C@{\phantom{-\alpha\lambda-}}
				C!{\phantom{-\alpha}}
				C!{\phantom{\lambda-}}
				C@{\phantom{-\alpha}}
				R}
				         \lambda{-}3{-}2\alpha 
				&          1                                              &                                                                       
					&        &                                                         &                                \\
				\phantom{\lambda -}1           
				& \alignas{1}{c}{\phantom{2}\lambda{-}2\phantom{\lambda}} &          1                                                            
					&        &                                                         &                                \\
				&          1                                              & \alignas{1}{c}{\phantom{\alpha}\lambda{-}3{-}\alpha\phantom{\lambda}} 
					& \Ddots &                                                         &                                \\
				&                                                         &                                                                       
					& \Ddots &                                                         &                                \\
				&                                                         &                                                                       
					& \Ddots & \alignas{1}{c}{\phantom{2}\lambda{-}2\phantom{\lambda}} &           1\phantom{- 2\alpha} \\
				&                                                         &                                                                       
					&        &          1                                              & \lambda{-}3        {-}2\alpha 
			\end{NiceArray}\right) \notag\\
	&= \det\left(
		\begin{NiceArray}[nullify-dots]{RCCCL:R!{\phantom{1}}CCL}
			\alignas{\lambda{-}3}{l}{\lambda{-}3{-}2\alpha} & & &   &   & 1 &                         &                          &   \\
			  &     \alignas{1}{c}{\lambda{-}3{-}\alpha}      & &   &   & 1 &                         &                          &   \\
			  &   &                   \Ddots                    &   &   &   &         \Ddots          &          \Ddots          &   \\
			  &   & &     \alignas{1}{c}{\lambda{-}3{-}\alpha}      &   &   &                         &                          & 1 \\
			  &   & & & \alignas{3{-}2\alpha}{r}{\lambda{-}3{-}2\alpha} &   &                         &                          & 1 \\
			\hdashline
			1 & 1 &                                             &   &   & \alignas{\lambda\fphantom{.5}{-1}}{l}{\lambda{-}2} & & &   \\
			  &   &           \Ddots                            &   &   &   &                      \Ddots                      & &   \\
			  &   &                            \Ddots           &   &   &   & &                       \Ddots                     &   \\
			  &   &                                             & 1 & 1 &   & & & \alignas{\fphantom{.5}{1-}2}{r}{\lambda{-}2}     
		\end{NiceArray}\right) \notag\\
	&= \dfrac{1}{\lambda{-}2} \cdot \det\left(
		\begin{NiceArray}[nullify-dots]{LCCCR}
			         b  - \Delta & -1 &        &    &                             \\
			\phantom{b} - 1      &  b & \Ddots &    &                             \\
			                     &    & \Ddots &    &                             \\
			                     &    & \Ddots &  b & {} - \alignas{\Delta}{l}{1} \\
			                     &    &        & -1 & b  -          \Delta 
		\end{NiceArray}\right)\sqrdim{t}, \label{eq: matrix T}
\shortintertext{where}
	\intercenter{2}{
		 b = (\lambda{-}2)(\lambda{-}3{-}\alpha) - 2, \qquad \Delta = \alpha(\lambda{-}2) - 1,
	} \notag
\end{align}
and Equation~\eqref{eq: matrix T} is obtained by Schur's determinant
identity. Note that both $b$ and~$\Delta$ depend on $\lambda$.

Denote the matrix in Equation~\eqref{eq: matrix T} by~$\mat{T}$. To find the roots of equation ${\det}\mat{T} = 0$, let
\begin{equation}
	\label{eq: g(lambda, theta)}
	\operatorname{g}(\lambda,\theta) \defas \sin{(t{+}1)}\theta - 2\Delta\sin{t\theta} + \Delta^2\sin{(t{-}1)\theta},
\end{equation}
and $\mu_l(\lambda),\ l = 1,2,\dotsc,t$ be eigenvalues of matrix~$\mat{T}$, viewing~$\lambda$ as a parameter. When $\lambda\>=4$, $\mat{T}$ is real symmetric, by Lemma~\ref{lem: eigs of tridiagonal matrices}, we have
\begin{gather}
	{\det}\mat{T} = \prod_{l=1}^t \mu_l(\lambda) = \prod_{l=1}^t (b - 2\cos\theta_l), \label{eq: factorization of |mat{T}|}
\intertext{where $ \theta_l \equiv \theta_l(\lambda) \in \R$ is given by}
	\begin{dcases}
		\operatorname{g}(\lambda,\theta_l) = 0, & \text{ if } \sin\theta_l \neq 0 \text{ and } \lambda > 4; \\
		\cos\theta_l = 1,                       & \text{ if } \sin\theta_l   =  0 \text{ and } \lambda = 4. 
	\end{dcases} \notag
\end{gather}
Without loss of generality, {we can restrict $\theta_l \in (0,\pi)$ as we are considering $\lambda>4$.}

With Equation~\eqref{eq: factorization of |mat{T}|}, we factorized
${\det}\mat{T}$, which is essentially a polynomial of~$\lambda$ of
order~$2t$. {Thus the number of our desired~$\lambda$ can
  be obtained by counting roots from lower order polynomials. Formally
  speaking, we solve for
\begin{gather*}
	\mu_l(\lambda) = b(\lambda) - 2\cos\theta_l = 0,
\shortintertext{with restriction}
	\operatorname{g}(\lambda,\theta_l) = 0, \quad \theta_l \in (0,\pi),
\end{gather*}
and count all the roots greater than~$4$.}

\item find~$t$ number of $\lambda(H_{1,\infty})$'s that are greater than~$4$.

Differentiating $\mu_l(\lambda)$ gives
\[
	\frac{\d\mu_l}{\d\lambda} = \lambda{-}3 + \sqrt{\lambda(\lambda{-}4)} + \frac{2}{\sqrt{\lambda(\lambda{-}4)}} + 2\sin\theta_l\frac{\d\theta_l}{\d\lambda}.
\]
Note that when $\lambda > 4$, the left hand side of the above equation are all positive except for the last term. Also note that 
\begin{align*}
	\mu_l(4)										&= -2 - 2\cos\theta_l(4) < 0, \\
	\mu_l\inparen*{ 2{\textstyle\sqrt{2}{+}2)} }	&= 2 - 2\cos\theta_l(2\sqrt{2}{+}2) > 0.
\end{align*}

In the following we will consider function $\theta_l(\lambda) \in
(0,\pi)$ determined by $\operatorname{g}(\lambda,\theta_l(\lambda)) =
0$, and show that $\theta_l'(\lambda) > 0$, thus exactly $t$~number of
$\lambda_j(H_{1,\infty})$'s that are strictly greater than~$4$ can be
found. Observe  that
\begin{align*}
	\operatorname{g}(\lambda,\theta) 
		&= \Im\inparen*{\e^{\i(t{+}1)\theta} - 2\Delta\e^{\i t\theta} + \Delta^2\e^{\i(t{-}1)\theta}} \\
		&= \Im\inparen*{\e^{\i t\theta} \inparen*{\e^{\i\theta} - 2\Delta + \Delta^2\e^{-\i\theta}} \vphantom{\Delta^2\e^{\i(t{-}1)\theta}}}.
\end{align*}
Let 
\begin{equation} \label{eq: omega}
	\rho\e^{\i\omega} \equiv \inparen*{\e^{\i\theta} - 2\Delta + \Delta^2\e^{-\i\theta}},
\end{equation} 
where $\rho$ is a positive number (obviously it cannot be~$0$ since $\Delta \in \R$ when $\lambda > 4$) and $\omega \in [0,2\pi)$ is some function of $\Delta$ and~$\theta$. Then
\[
	\operatorname{g}(\lambda,\theta) = \Im\inparen*{\e^{\i t\theta} \cdot \rho\e^{\i\omega}} = \rho\sin(t\theta + \omega).
\]
So $\operatorname{g}(\lambda,\theta_l) = 0$ implies $t\theta_l + \omega = \Z\pi$. Differentiating this equation with respect to~$\lambda$, we obtain
\begin{gather*}
	t\frac{\d\theta_l}{\d\lambda} + \frac{\partial\omega}{\partial\theta_l}\frac{\d\theta_l}{\d\lambda} 
	+ \frac{\partial\omega}{\partial\Delta}\frac{\d\Delta}{\d\lambda} = 0,
\shortintertext{i.e.,}
	\frac{\d\theta_l}{\d\lambda} = -\frac{\dfrac{\partial\omega}{\partial\Delta}\dfrac{\d\Delta}{\d\lambda}}{t + \dfrac{\partial\omega}{\partial\theta_l}}.
\end{gather*}
We already know that $\dfrac{\d\Delta}{\d\lambda} < 0$. To get the signs of $\dfrac{\partial\omega}{\partial\Delta}$ and $\dfrac{\partial\omega}{\partial\theta_l}$, let us go back to Equation~\eqref{eq: omega}, and expand it into real and imaginary part, namely
\begin{align}
	\rho\cos\omega	&= \inparen*{1+\Delta^2}\cos\theta - 2\Delta, \label{eq: omega_re} \\
	\rho\cin\omega	&= \inparen*{1-\Delta^2}\cin\theta.           \label{eq: omega_im}
\end{align}
Note that by Equation~\eqref{eq: omega_im}, $\sin\omega > 0$, and partially differentiating Equation~\eqref{eq: omega_re} with respect to $\Delta$ and~$\theta$, we have
\begin{align*}
	\frac{\partial\omega}{\partial\Delta}
		&= -\frac{2(\Delta\cos\theta - 1)}{\rho\sin\omega} > 0, \\
	\frac{\partial\omega}{\partial\theta}
		&= \frac{\alignas{-2(\Delta\cos\theta - 1)}{c}{\inparen*{1+\Delta^2}\sin\theta}}{\phantom{-}\rho\sin\omega} > 0.
\end{align*}
Hence $\theta_l'(\lambda) > 0$. The claim~\eqref{eq: claim} is thus
established and the proof is complete. \qedhere
\end{enumerate}
\end{proof}

\subsection{Uniform trees with trunk length\/ $m = 0$}
As shown by the numerical results in 
Table~\ref{tab: m=0}, this case is quite different of the previous two
situations.

\begin{thm} \label{thm: m=0}
  Consider a uniform tree  with trunk length $m=0$.  
  We have
  \[
	\n4+(H_{0,k}) \<= \floor*{\frac{\nT(H_{0,k}) + 1}{2}}, \quad \forall k \>= 1,
  \]
  where~$[x]$ is the largest integer no greater than~$x$. Moreover,
  the equality is achieved when\/~$k$ is large enough, i.e.,  branchs
  are long enough.
\end{thm}

\begin{proof}
  We proceed in two steps.
\begin{enumerate}[label = \textbullet\ \emph{Step} \arabic*:, wide = 0pt, listparindent = \parindent, itemsep = \medskipamount]
\item prove the upper bound by induction on~$\nT$.
\begin{itemize}[label = $\lozenge$, labelsep = *, leftmargin = 0pt, itemindent = \parindent - (\widthof{\itshape{B\/}} - \widthof{\itshape{B}}), listparindent = \parindent]
\item \emph{Base Case with\/ $\nT=2$}
	
    When $\nT = 2$, let $\widetilde{H} = (V, E {\setminus}
    \inbrace{e})$, where $e$ connects the only two T-junctions
    of~$H_0$. Then $\widetilde{H}$ consists of two paths, say
    $\widetilde{P}_1$ and~$\widetilde{P}_2$, respectively $\bigl($see
    Figure~\ref{fig: m=0;nT=2,3}(a)$\bigr)$. By Equation~\eqref{eq: path graph eigs},
    $ \max\bigl\{ \lambda_1(\widetilde{P}_1), \lambda_1(\widetilde{P}_2) \bigr\}< 4$. By Lemma~\ref{lem: interlacing theorem},
    $\lambda_2(H_0) \<= \lambda_1(\widetilde{H}) = \max\bigl\{
    \lambda_1(\widetilde{P}_1), \lambda_1(\widetilde{P}_2) \bigr\} <
    4$. Since one can cut either of the branches in~$H_0$ to produce a starlike tree, by Lemma~\ref{lem: interlacing theorem} this means $\lambda_1(H_0) \>= 4$. Hence we have $\n4+ = 1 \<= \floor{(\nT{+}1)/2}$. 

\begin{figure} 
	\centering
	\includegraphics[width = .9\textwidth]{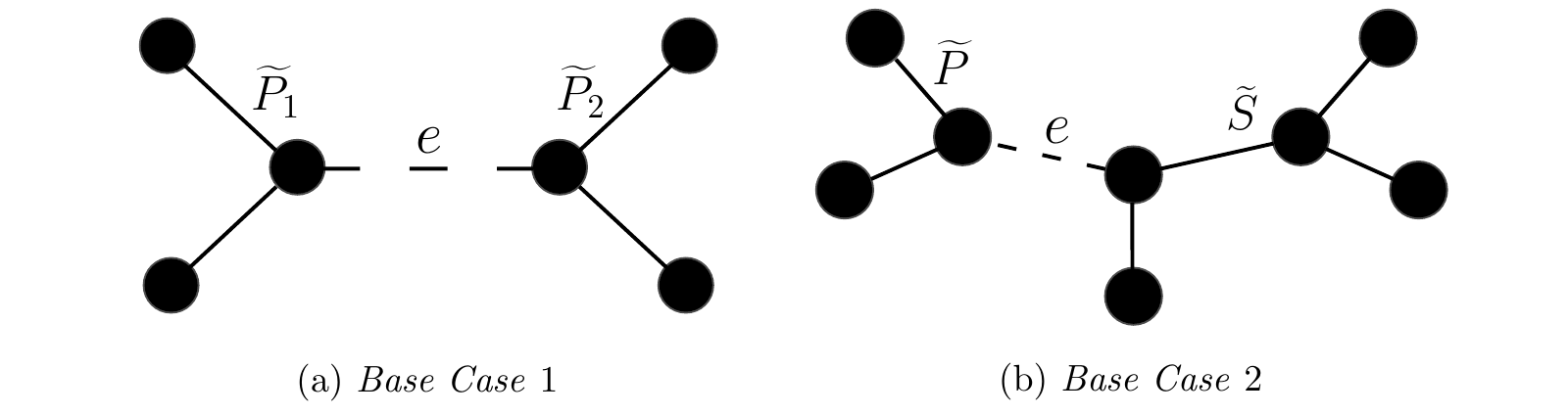}
	\caption{Two examples of base cases in the proof of Theorem~\ref{thm: m=0}.}
	\label{fig: m=0;nT=2,3}
\end{figure}

\item \emph{Base Case with\/ $\nT=3$}

When $\nT = 3$, let $\widetilde{H} = (V, E {\setminus} \inbrace{e})$, where $e$ connects either pair of the consecutive T-junctions of~$H_0$. Then $\widetilde{H}$ consists of one path and one starlike tree, say $\widetilde{P}$ and~$\widetilde{S}$, respectively $\bigl($see Figure~\ref{fig: m=0;nT=2,3}(b)$\bigr)$. Then $\lambda_1(\widetilde{P}) < 4$ by Equation~\eqref{eq: path graph eigs}, $\lambda_1(\widetilde{S}) > 4 > \lambda_2(\widetilde{S})$ by Lemma~\ref{lem: starlike tree}. Furthermore, by Lemma~\ref{lem: interlacing theorem}, $\lambda_1(H_0) \>= \lambda_1(\widetilde{H}) = \max\bigl\{ \lambda_1(\widetilde{P}), \lambda_1(\widetilde{S}) \bigr\} = \lambda_1(\widetilde{S}) > 4$, and $\lambda_3(H_0) \<= \lambda_2(\widetilde{H})= \max\bigl\{ \lambda_1(\widetilde{P}),\lambda_2(\widetilde{S}) \bigr\} < 4$. Hence, we have $\n4+ \<= 2 \<= \floor{(\nT{+}1)/2}$. 

\item \emph{Induction\/}
\quad
Suppose $\n4+ \<= \floor{(\nT{+}1)/2}$ holds for all $\nT \<= i$, for some $i \>= 2$.

When $\nT = i{+}1$, let $\widetilde{H} = (V, E {\setminus}
\inbrace{e})$, where $e$ connects either of the T-junctions at one end
of~$H_0$ and its consecutive T-junction. Then $\widetilde{H}$ consists
of a path and a tree, say $\widetilde{P}$ and~$\widetilde{H}_1$,
respectively. Note that $\widetilde{H}_1$ has $\nT{-}2$ T-junctions,
and thus at most $\floor{(\nT{-}2{+}1)/2} = \floor{(\nT{-}1)/2}$ eigenvalues
are no smaller than~$4$. Since all the eigenvalues of~$P$ are less than~$4$, by Lemma~\ref{lem: interlacing theorem}, $\n4+ \<= \floor{(\nT{-}1)/2} + 1 = \floor{(\nT{+}1)/2}$.	
\end{itemize}

\item prove the condition to achieve the upper bound.
  
  We will use $t$ instead of $\nT$ for simplicity. Following a
  procedure similar to the one used in the proof of Theorem~\ref{thm: m=1} with $k \rightarrow \infty$, we get
\begin{align*}
	{\lim_{k\rightarrow\infty}}\det\inparen*{\lambda\mat{I} - \mat{L}_0}
		&= \det\inparen*{\lambda\mat{I} - \B_~} \cdot \det\left(
			\begin{NiceArray}[nullify-dots]{LLCRR}
				\alignas{\lambda{-}3}{l}{\lambda{-}3{-}2\alpha} & \phantom{{-}1\alpha}1 &        &                                                                      &   \\
				1 &            \alignas{111}{c}{\lambda{-}3{-}\alpha}                   & \Ddots &                                                                      &   \\
				  &                                                     \Ddots          & \Ddots &     \Ddots                                                           &   \\
				  &                                                                     & \Ddots & \alignas{111}{c}{\lambda{-}3{-}\alpha}                               & 1 \\
				  &                                                                     &        & 1\phantom{\lambda{-}1} & \alignas{3{-}2\alpha}{r}{\lambda{-}3{-}2\alpha}
			\end{NiceArray}\right)\sqrdim{t} \\
		&= \det\inparen*{\lambda\mat{I} - \B_~} \cdot \prod_{l=1}^t \mu_{0,l}(\lambda),
\end{align*} 
where $\mat{L}_0$ is a short for $\mat{L}(H_0)$, $\B_~$ is just the one defined previously in the proof of Theorem~\ref{thm: m=1}, and
\[
	\mu_{0,l}(\lambda) \defas (\lambda{-}3{-}\alpha) + 2\cos\theta_{0,l},
\]
with also the same~$\alpha$ defined previously, and $\theta_{0,l} \equiv \theta_{0,l}(\lambda) \in \R$ is given by
\begin{gather*}
	\begin{dcases}
		\operatorname{g_0}(\lambda,\theta_{0,l}) = 0,	& \text{ if } \sin\theta_{0,l} \neq 0 \text{ and } \lambda > 4;  \\
		\cos\theta_{0,l} = -1              ,         	& \text{ if } \sin\theta_{0,l}   =  0 \text{ and } \lambda = 4,
	\end{dcases}
\shortintertext{where}
	\operatorname{g_0}(\lambda,\theta) \defas \sin{(t{+}1)}\theta + 2\alpha\sin{t\theta} + \alpha^2\sin{(t{-}1)\theta}.
\end{gather*}

{Similarly restricting $\theta_{0,l} \in (0,\pi)$ as before.} Differentiating $\mu_{0,l}(\lambda)$ with respect to~$\lambda$ gives
\[
	\frac{\d\mu_{0,l}}{\d\lambda} = \frac{\lambda-2}{2\sqrt{\lambda(\lambda{-}4)}} + \frac{1}{2} - 2\sin\theta_{0,l}\frac{\d\theta_{0,l}}{\d\lambda}.
\]
In contrast to the proof of Theorem~\ref{thm: m=1}, this time we expect $\theta_{0,l}'(\lambda) < 0$. To prove this, let 
\[
	\rho_0\e^{\i\omega_0} \equiv \inparen*{\e^{\i\theta} + 2\alpha + \alpha^2\e^{-\i\theta}},
\]
from which one can similarly derive
\begin{align*}
	\frac{\partial\omega_0}{\partial\alpha}
		&= -\frac{2(\alpha\cos\theta + 1)}{\rho_0\sin\omega_0} < 0, \\
	\frac{\partial\omega_0}{\partial\theta}
		&= \frac{\alignas{-2(\Delta\cos\theta - 1)}{c}{\inparen*{1+\alpha^2}\sin\theta}}{\phantom{-}\rho_0\sin\omega_0} > 0,
\end{align*}
and $\operatorname{g_0}(\lambda,\theta_{0,l}) = 0$ implies
\begin{equation} \label{eq: to find theta_l}
	t\theta_{0,l} + \omega_0 = \Z\pi.
\end{equation}
Differentiating the above equation with respect to~$\lambda$ eventually leads to 
\[
	\frac{\d\theta_{0,l}}{\d\lambda} = -\frac{\dfrac{\partial\omega_0}{\partial\alpha}\dfrac{\d\alpha}{\d\lambda}}
		{t + \dfrac{\partial\omega_0}{\partial\theta_{0,l}}} < 0.
\]

However, similar to  the proof of  Theorem~\ref{thm: m=1},  we here have 
\begin{gather*}
	\mu_{0,l}\Bigl(\frac{16}{3}\Bigr) = 2\cos\theta_{0,l}\Bigl(\frac{16}{3}\Bigr) + 2 > 0,
\shortintertext{but}
	\mu_{0,l}(4) = 2\cos\theta_{0,l}(4) = 2\cos\inparen*{\frac{l\pi}{t}}. 
\end{gather*}
So $\mu_{0,l}(4) < 0$ only if $t/2 < l < t$, and it is trivial to see that there are exactly $\floor{(t{+}1)/2}$ number of such~$l$'s, which completes our proof. \qedhere
\end{enumerate}
\end{proof}

\section{Analysis of dendrites of RGCs}\label{sec:RGC}
In the previous sections, we have analyzed in detail how 
characteristic numbers such as T-junction numbers, branch and trunk
lengths impact on the number of spike eigenvalues of the graph
Laplacian of the class of uniform trees. The purpose of this section is to show how these
theoretical results on uniform trees can help us understand the phase
transition observed in dendrite graphs as shown in Figure~\ref{fig: RGC & eigs distributions}.

First Table~\ref{tab: RGC statistics} gives a few summary  statistics
on different characteristics of the two dendrite graphs shown in
Figure~\ref{fig: RGC & eigs distributions}(a)-(b).

\begin{table}
	\centering
	\caption{Summary statistics for the two dendrite graphs.
      Upper block\textup: number of eigenvalues\/ ${\>=}4$ as well as vertices with degree\/ ${=}3$ and\/ ${\>=}4$. Middle block\textup: distribution of trunk lengths~$m$. Lower block\textup: statistics of branch lengths\/~$k$.}
	\begin{tabular}{crr}
								& RGC~\#100	& RGC~\#60	\\
	\midrule											
		$\n4+$					& 12		& 138		\\
		$\nT$					& 12		& 166		\\
		$\#\{\deg(v) \>= 4\}$	& 0			& 7			\smallskip\\
	\midrule											
	\multicolumn{3}{l}{\% of trunks with different lengths} \\
		$m=0$					& 0\%		& 27.3\%	\\
		$m=1$					& 0\%		& 8.7\%		\\
		$m\>=2$					& 100\%		& 64.0\%	\smallskip\\
	\midrule											
	\multicolumn{3}{l}{branch length statistics}		\\
		min						& 15		& 1			\\
		max						& 132		& 125		\\
		mean					& 64.8		& 20.7		\\
		median					& 64		& 13		\\
	\bottomrule
	\end{tabular}
	\label{tab: RGC statistics}
\end{table}

We proceed as follows.
\begin{enumerate}
\item Each dendrite graph is assimilated as a mixture of uniform trees
  $H_{m,k}$ with different trunk and branch lengths $(m,k)$;
\item
  Using the theoretical results established previously on the number
  of spike eigenvalues for uniform trees, we propose an estimate, say
  $\hatn4+$, for
  the observed number of spike eigenvalues of the dendrite graph, say $\n4+$;
\item
  By confronting these two numbers, the meaningness of uniform trees
  as a model for dendrite graphs can be assessed.
\end{enumerate}

To start with, it is important to note the following features in a
dendrite graph which do not exist in a uniform tree:
\begin{enumerate}
	\item[(d1)]	Branch lengths may vary a lot in a real dendrite graph; 
    \item[(d2)]	Trunks that incident to the same T-junction may have different lengths;
	\item[(d3)]	Not all T-junctions lie along a line, i.e., {there may be T-junctions whose neighboring vertices are all T-junctions}; 
	\item[(d4)]	Not all junctions are T-junctions, i.e., there may be junctions of degree larger than 3.
\end{enumerate}

\begin{figure}
	\centering
	\includegraphics[width = \textwidth]{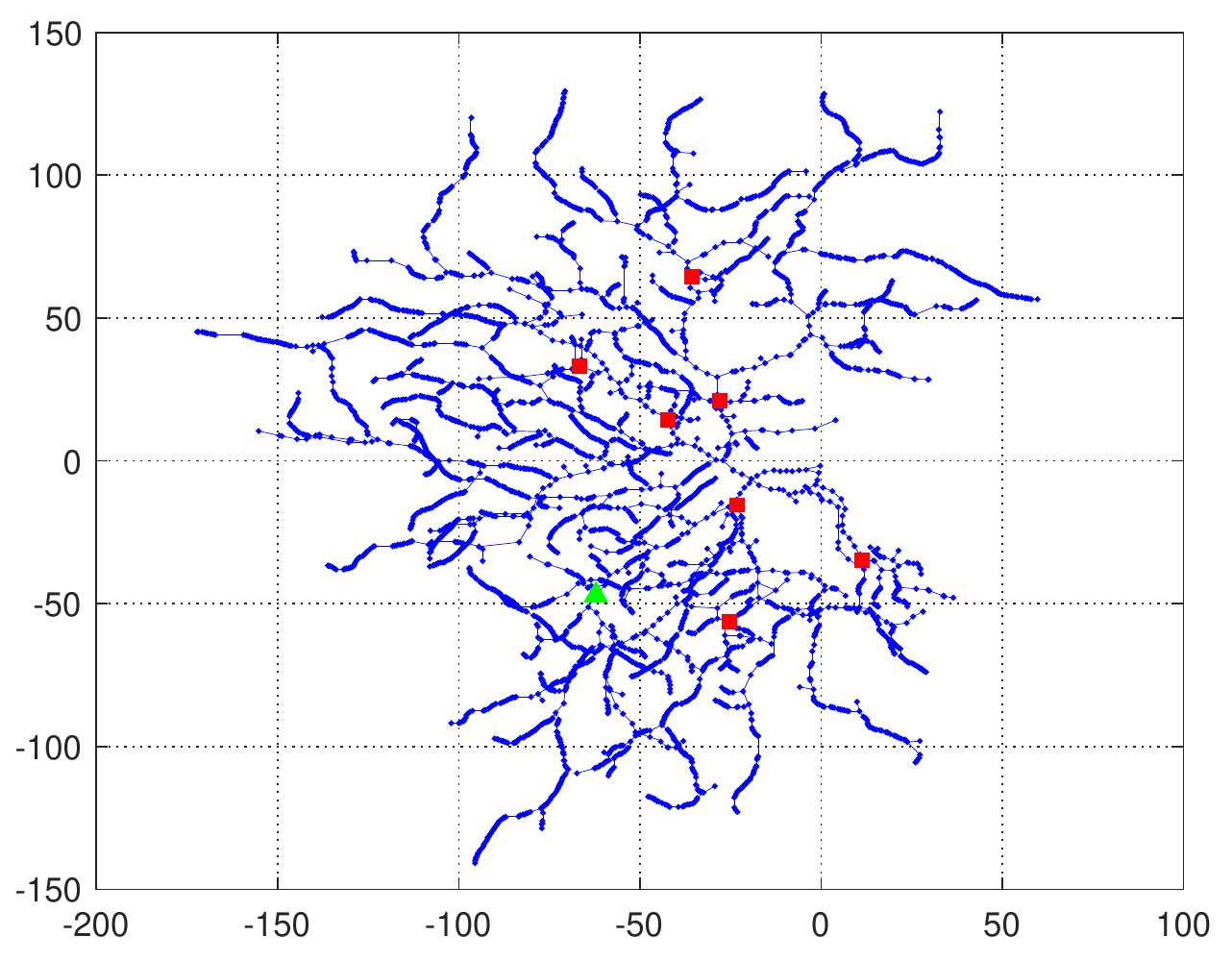}
	\includegraphics[width = .497\textwidth]{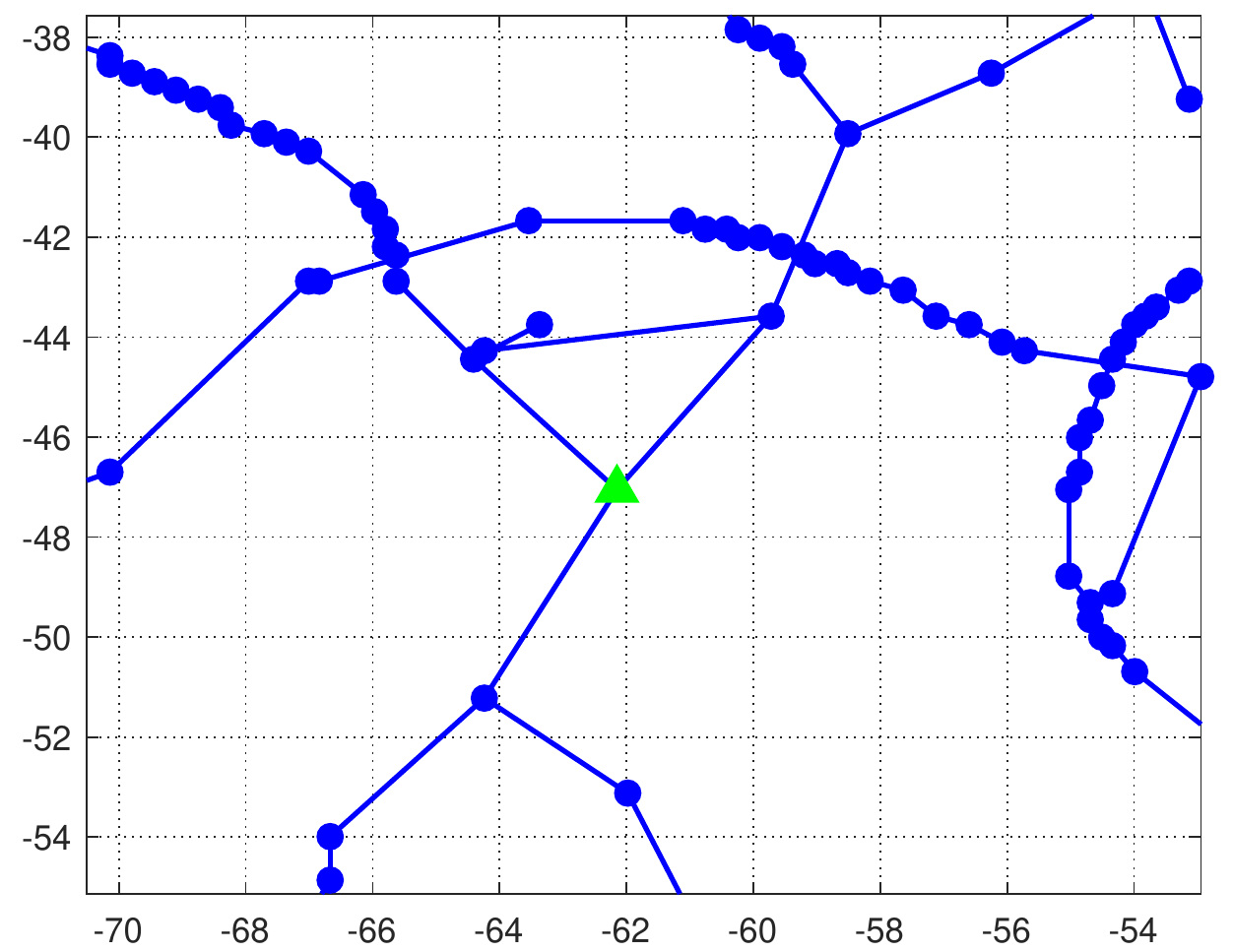}
	\includegraphics[width = .493\textwidth]{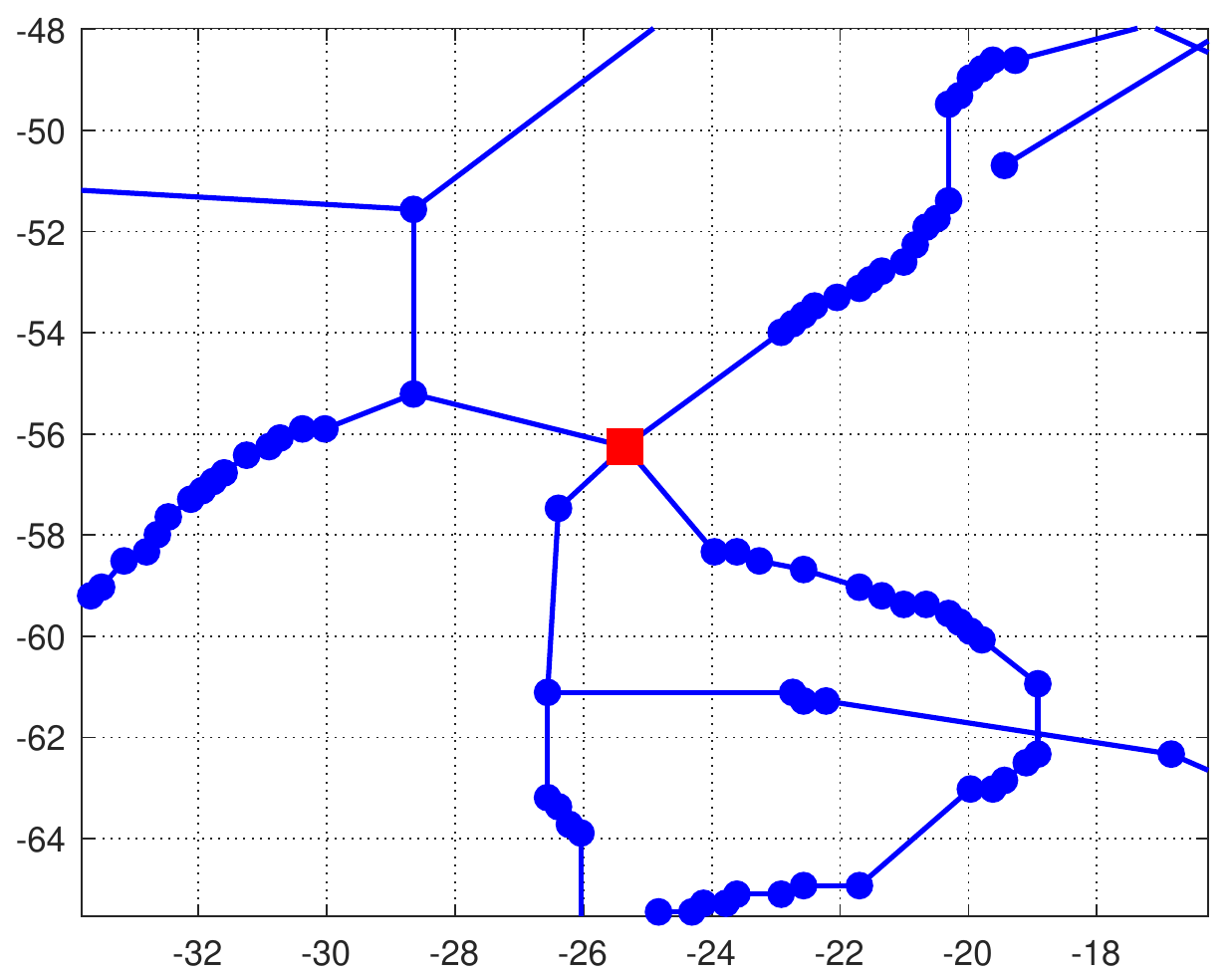}
	\caption{Upper\textup: dendrite graph of RGC $\#60$ with marked vertices corresponding to features \textup{(d3)(}green triangle\textup) and \textup{(d4)(}red square\textup). Lower left\textup: zoom up region around feature \textup{(d3)}. Lower right\textup: zoom up region around the bottom feature \textup{(d4)}.}
	\label{fig: marked dendrite tree}
\end{figure}

{Example of features (d3) and (d4) are given in Figure~\ref{fig: marked dendrite tree}.}

Note that from our previous discussion, {$\n4+(G)$ is
non-decreasing for any dendrite graph~$G$ if any of its branch is
extended by one unit. Therefore,} as long as branches are long enough,
{the limiting} results derived for uniform trees still hold even though branch
lengths are not constant anymore (feature (d1)).
Note  also from Table~\ref{tab: RGC statistics}, very few vertices in
those dendrite graphs have degree greater than~$3$ (feature (d4)).
Finally, to deal with the remaining features (d2)-(d3), we propose
the following method of estimation for the number of spike eigenvalues
$\n4+$.

\smallskip
{\bfseries Estimating  the number of spike eigenvalues $\mathbf{n_{4^{\!+}}}$.} \quad
Let  $G = (V,E)$ be a dendrite graph and consider the set of junction vertices
\[  J:= \bigl\{  v \in V|\deg(v) \>= 3\bigr\}.
\]
\begin{enumerate}
\item[(i)] For each junction vertex $v\in J$, find all trunks incident at
  $v$ (recall that a trunk is a path connecting two junction
  vertices).   Let $m(v)$ be the minimum of the lengths of these
  incident trunks.
  Next classify $v$  into one of the three groups
  $\inbrace{J_0,J_1,J_{2^+\!}}$ according to  the conditions
  $m(v)=0$, $m(v)=1$ and $m(v) \>= 2$, respectively.
  
\item[(ii)]  The estimate for the number of spike eigenvalues $\n4+$ is 
\begin{equation}
    \label{eq: estimate}
	\hatn4+ = \floor*{\frac{n_0+1}{2}} + n_1 + n_{2\scalebox{0.69444}{$^+$}},
\end{equation}
where $(n_0, n_1, n_{2\scalebox{0.69444}{$^+$}}) \defas (\#J_0,\#J_1,\#J_{2^+})$.
$\qed$
\end{enumerate}  

Note that the formula in \eqref{eq: estimate} is
directly inspired by the theoretical results for uniform trees
presented in Theorems~\ref{thm: m>=2}, \ref{thm: m=1} and \ref{thm: m=0}.
Applying this estimator to the two dendrite graphs shown on Figure~\ref{fig: RGC & eigs distributions},
we obtain 
$\hatn4+ = 12$ and~$134$ for RGC~\#100 and RGC~\#60, respectively.
These predicted values are very close to their true value $\n4+$ in Table~\ref{tab: RGC statistics},
namely 
$\n4+ = 12$ and~$138$ for RGC~\#100 and RGC~\#60, respectively.
In conclusion, viewing a dendrite graph as a mixture of uniform trees
of different types can help one predict the number of spike eigenvalues
of the graph Laplacian.

\section*{Acknowledgments}
We sincerely thank Professor N. Saito for providing us with data of
retinal ganglion cells $\#60$ and~$\#100$ used in this paper and the
Matlab code for drawing Figure~\ref{fig: RGC & eigs distributions}.
We are also grateful to Ms.  Yihui Ma who contributed to  a preliminary
draft of the paper.

\end{document}